\documentclass[12pt]{amsart}
\usepackage{amsmath,amsthm,latexsym,amscd,amsbsy,amssymb,url}
\usepackage[all]{xypic}
 \relax

\chardef\bslash=`\\ 

\makeatletter
\def\verbatim{\interlinepenalty\@M \@verbatim
  \leftskip\@totalleftmargin\advance\leftskip2pc
  \frenchspacing\@vobeyspaces \@xverbatim}
\makeatother
\newtheorem{thm}{Theorem}[section]
\newtheorem{cor}[thm]{Corollary}
\newtheorem{lem}[thm]{Lemma}

\begin{document}


\title
{On Hereditarily Normal Topological Groups}
\author{Raushan ~Z.~Buzyakova}
\email{Raushan\_Buzyakova@yahoo.com}
\keywords{topological group, hereditary normal space}
\subjclass{54H11, 22A05, 54D15}


\begin{abstract}{
In this paper we investigate 
hereditarily normal topological groups and their subspaces.
We prove that every compact subspace 
of a hereditarily
normal topological group is metrizable. To prove this statement we first show
that a hereditarily normal topological group with a non-trivial convergent sequence
has $G_\delta$-diagonal. This implies, in particular, that every countably compact subspace of a hereditarily normal topological
group with
a non-trivial convergent sequence is metrizable.  Another corollary is that under the Proper Forcing Axiom,
every countably compact subspace of a  hereditarily normal topological group is metrizable.
}
\end{abstract}

\maketitle
\markboth{Raushan Z. Buzyakova}{On Hereditarily Normal Topological Groups}
{ }

\section{Introduction}\label{S:intro}
It is a known fact that a hereditarily normal compact topological group is metrizable. This fact follows from the theorem of R. Engelking \cite{Eng1}  that every compact topological group contains a subspace homeomorphic to $\{0,1\}^\tau$, where $\tau$ is the weight of the group.  It is also a corollary to deep work of other mathematicians. For a proof of this theorem  and historical development around it we refer to \cite[Theorem 4.2.1]{AT} and \cite{Sh}. In this paper we show that
not only compact hereditarily normal topological groups are metrizable but any compact subset of any hereditarily normal topological group is metrizable as well (Theorem \ref{thm:compactainT5groups}). Thus, in the class of compact spaces, only metrizable ones  can be embedded into  hereditarily normal topological groups.

After it was established that every compact hereditarily normal topological group is metrizable 
it was natural to wonder if compactness could be relaxed to countable compactness. The example of Hajnal and Juh\'asz in \cite{HJ} closed this door in ZFC. More precisely, assuming the Continuum Hypothesis, Hajnal and Juh\'asz constructed a hereditarily normal hereditarily separable countably compact topological group which is not compact, hence  not metrizable. The example has many additional features and one of them is that it lacks non-trivial
convergent sequences. In \cite[Theorem 3]{Eis}, T. Eisworth showed  that this feature is not an accessory but a necessity. Eisworth's result is a corollary to one of our main theorems.  Namely, we prove that
every hereditarily normal topological group with a non-trivial convergent sequence has a $G_\delta$-diagonal 
(Theorem \ref{thm:T5groupwithsequence}). This result and the theorem of Chaber \cite{Ch} imply that every countably compact subset of a   hereditarily normal topological group with a non-trivial convergent
sequence is metrizable.  
In \cite{NSV}, Nyikos, Soukup, and Velickovic proved that under  the Proper Forcing Axiom, every countably compact hereditarily normal space is sequentially compact. This theorem and the mentioned 
 Theorem \ref{thm:T5groupwithsequence} of this paper imply  that under the Proper Forcing Axiom,
every countably compact subset of a  hereditarily normal topological group is metrizable. This, in its turn, implies another earlier result of Eisworth \cite[Corollary 10]{Eis} that under PFA, every countably compact hereditarily
normal topological group is metrizable.
 We would like to mention that this work was inspired by the well-known necessity condition of Katetov
for hereditary normality of the product of two spaces \cite[Theorem 1]{Kat}. 

In notation and terminology we will follow \cite{Eng}. We reserve the symbol $\star$ for the binary group operation of a group $G$ and letter $e$ for the neutral element of $G$. Following a group theory convention, we will omit the group binary operation symbol in standard situations.
In particular, for elements $a,b\in G$ we will write $ab$ instead of $a\star b$. However, in a few places in this paper, the use of the symbol
$\star$ will be necessary for the sake of clarity and in some cases to stress one's attention to the  relation of Katetov's argument to our work.

A space $X$ has {\it a $G_\delta$-diagonal} if the diagonal 
$\{\langle x,x\rangle: x\in X\}$ is the intersection of a countable family of its open neighborhoods in $X\times X$. {\it A  non-trivial convergent sequence} is a space homeomorphic to the subspace of the reals $\{0\}\cup \{1/n:n=1,2,3,...\}$.
Since we will often switch our attention from a given space  to its subspaces and vice versa we agree that  when dealing with a space $X$, its subspace $Y$, and subsets $S\subset X$ and $P\subset Y$, by $\bar S$ we denote the
closure of $S$ in $X$ and by $cl_Y(P)$ the closure of $P$ in $Y$. All spaces are assumed to be $T_1$.

\section{Results}\label{S:results}

\par\bigskip
For our argument we will need the following extract from Katetov's proof of his theorem in \cite{Kat}. For convenience, we also give a sketch of  Katetov's proof with some notational changes that fit our goal.
\par\bigskip\noindent
\begin{thm}\label{thm:katetov}{\rm ({\bf Katetov \cite[extract from Theorem 1]{Kat}})}
Let $T$ be a topological space and let $t\in T$ have uncountable pseudocharacter in $T$.
Also let $S$ be a separable topological space and let $s\in S$ be a limit point of $S$.
Then the sets $A = [S\times \{t\}]\setminus \{\langle s,t\rangle \}$ and 
$B = [\{s\}\times T]\setminus \{\langle s,t\rangle\}$
are closed and disjoint sets in $Z= [S\times T]\setminus \{\langle s,t\rangle \}$ that cannot be separated by open
neighborhoods in $Z$. 
\end{thm} 
\begin{proof} ({\it Follows  Katetov's argument.}) Closedness and disjointness are clear. Let $U$ be a neighborhood
of $A$ in $Z$. We need to show that $cl_Z(U)$ meets $B$. Fix a countable set $D\subset S$ which is dense in $S$.
For each $x\in D\setminus \{s\}$, 
fix a neighborhood $U_x$ of $t$ in $T$ such that
$\{x\}\times U_x\subset U$. Since $t$ has uncountable pseudocharacter, we conclude that 
there exists $y\in \bigcap \{U_x: x\in D\setminus \{s\}\}$ distinct from $t$.
This means that $[D\setminus \{s\}]\times \{y\}$ is in $U$. Since $D$ is dense in $S$, 
we conclude that $\langle s,y\rangle\in cl_Z(U)\cap B$.

\end{proof}

\par\bigskip\noindent
\begin{lem}\label{lem:pseudocharacteratneutral}{\rm ({\bf Folklore})}
Let $G$ be a topological group. If the diagonal $\{\langle g,g\rangle :g\in G\}$  has uncountable pseudocharacter
in $G\times G$, then the neutral element $e$ of $G$ has uncountable pseudocharacter in $G$.
\end{lem}
\begin{proof}
The conclusion follows from the fact that $\star^{-1}(e) = \{\langle g,g^{-1}\rangle :g\in G\}$ and the set on the right
has the same pseudocharacter in $G\times G^{-1}$ as the diagonal $\{\langle g,g\rangle :g\in G\}$ in $G\times G$.
\end{proof}

\par\bigskip\noindent
\begin{thm}\label{thm:T5groupwithsequence}
A hereditarily normal topological group with a non-trivial convergent sequence has $G_\delta$-diagonal.
\end{thm}
\begin{proof}
Let $G$ be a topological group with a non-trivial convergent sequence.
Assume that the diagonal $\{\langle g,g\rangle :g\in G\}$ is not a $G_\delta$-set in $G\times G$. We need to
show that $G$ is not hereditarily normal. By Lemma \ref{lem:pseudocharacteratneutral}, the pseudocharacter of 
the neutral element $e$ is uncountable.
Let $\{e_n:n\in \omega\}$ be a sequence that converges to $e$ such that $e_n\not = e$ for every $n\in \omega$.
Such a sequence exists due to homogeneity of $G$ and the theorem's hypothesis.
For every $n\in \omega$, select an open neighborhood $U_n$ of $e$ whose closure does not meet $\{e_n, e_n^{-1}\}$.
Put $T = \bigcap_{n\in \omega} \overline U_n$. Since the pseudocharacter of $e$ is uncountable,
we conclude that $e$ is a limit point for the closed set $T$. Since $T$ is a $G_\delta$-set in $G$ and $\{e\}$ is not,
we conclude that $e$ has uncountable pseudocharacter in $T$.
Put $S=\{e\}\cup \{e_n:n\in \omega\}$. The following three sets are the key objects for 
the remainder of our argument:
$$
Z= (S\times T)\setminus \{\langle e,e\rangle\},\ A = \{\langle e_n,e\rangle:n\in \omega\},\ B = \{\langle e,y\rangle: y\in T\setminus \{e\}\}
$$ 
Since $e$ has uncountable pseudocharacter in $T$, by Katetov's theorem (Theorem \ref{thm:katetov}), $A$ and $B$ are closed and disjoint subsets
of $Z$ that cannot be separated by open sets in $Z$. 
To finish our proof, it suffices to show that
$\star (A)=\{e_n:n\in \omega\}$ and $\star (B)=T\setminus \{e\}$ are closed and disjoint subsets of $\star (Z)=ST\setminus \{e\}$.

Let us first prove that $\{e_n:n\in \omega\}$ is closed in $ST\setminus \{e\}$. For this observe that
$e$ is the only limit point of $\{e_n:n\in \omega\}$ that does not belong to $\{e_n:n\in \omega\}$. Since $e$ does not
belong to $ST\setminus \{e\}$ either, we conclude that $\{e_n:n\in \omega\}$ is closed in $ST\setminus \{e\}$.
The proof that $T\setminus \{e\}$ is closed in $ST\setminus \{e\}$ is analogous.
 
Now let us show that $\{e_n:n\in \omega\}$ and $T\setminus \{e\}$ are disjoint.  For this recall
that $T = \bigcap_{n\in \omega}\overline U_n$, where $\overline U_n$ misses $\{e_n\}$ for each $n$, which completes our proof.
\end{proof}
\par\bigskip
Chaber proved in \cite{Ch} that a countably compact space with a $G_\delta$-diagonal is metrizable.
Chaber's theorem and Theorem \ref{thm:T5groupwithsequence} imply the following statements.
\par\bigskip\noindent
\begin{cor}\label{cor:ccsubspaceinT5group} 
Every countably compact subspace of a hereditarily normal topological group that contains a non-trivial 
convergent sequence is
metrizable.
\end{cor}
\par\bigskip\noindent
\begin{cor}\label{cor:ccT5group}
{\rm (\cite[Corolary 10]{Eis}) } Every countably compact hereditarily normal  topological group that contains 
a non-trivial convergent sequence is metrizable.
\end{cor}
\par\bigskip
In \cite{NSV}, Nyikos, Soukup, and Velickovic proved that under  the Proper Forcing Axiom, every countably compact hereditarily normal  space is sequentially compact. This theorem and 
Corollary \ref{cor:ccsubspaceinT5group}  imply  the following statement.
\par\bigskip\noindent
\begin{cor}\label{cor:ccT5PFA}
Assume the Proper Forcing Axiom. Then
every countably compact subspace of a  hereditarily normal topological group is metrizable. 
\end{cor}
\par\bigskip
Next we shall prove that every compact subset of a hereditarily normal topological 
group is metrizable. We start with the following lemma.
\begin{lem}\label{lem:katetov}
Let $G$ be a hereditarily normal topological group and let $S$ and $T$ be its compact subspaces.
Suppose that $S$ is separable,  $s$ is a limit point of $S$, and $t\in T$ has uncountable character in $T$. Then
there exists a compactum $C\subset T$ such that $t$ has  uncountable character in $C$ and $sC\subset St$.
\end{lem}
\begin{proof}
Let us show that $C = \{g\in T: sg\in St\}$ is as desired. Since $S$ and $T$ are compact, we conclude that $C$ is compact. Clearly $t\in C$. It is left to show that $t$ has uncountable character in $C$.
We assume the contrary. Then there exists a countable family $\{U_n:n\in \omega\}$ of neighborhoods of $t$ 
in $T$ such that $F = \bigcap\{\overline U_n:n\in \omega\}$ misses $C\setminus \{t\}$. Since $\chi(t,T)$ is uncountable, $\chi(t,F)$ is uncountable as well.
To reach a contradiction it suffices to find $g\in F\setminus \{t\}$ such that $sg\in St$.

For this put $A=[S\times \{t\}]\setminus \{\langle s,t\rangle\}$ and $B=[\{s\}\times F]\setminus \{\langle s,t\rangle\}$. 
By Katetov's theorem (Theorem \ref{thm:katetov}), the sets $A$ and $B$ are closed and disjoint subsets of 
$Z=[S\times F]\setminus \{\langle s,t\rangle\}$ that cannot be separated by disjoint open neighborhoods in $Z$. 
Put $Z_1 = Z\setminus \star^{-1}(st)$. Clearly $Z_1$ is open in  $Z$ and contains both $A$ and $B$.
Therefore $A$ and $B$ cannot be separated by open neighborhoods in $Z_1$ either.
Since $G$ is hereditarily normal, 
the closures of $\star (A)=St\setminus \{st\}$ and $\star (B)=sF\setminus \{st\}$  in $\star (Z_1)\subset G\setminus \{st\}$ must meet.
Since $St$ is compact and $st\not \in G\setminus \{st\}$, we conclude that $St\setminus \{st\}$ is closed in $G\setminus \{st\}$.
The proof that $sF\setminus \{st\}$ is closed in $G\setminus \{st\}$ is analogous.
Since $St\setminus \{st\}$ and $sF\setminus \{st\}$ are closed in $G\setminus \{st\}$,
we conclude that  they meet. Therefore, we can find 
$g\in F\setminus \{t\}$ such that
$sg\in St$, which completes the proof.
\end{proof}

\par\bigskip
We are now ready to prove our main result. 
For reference we will first formulate an often-used  corollary from the following well-known result of Shapirovskii's Theorem \cite[Theorem 2]{Sha}, in which a {\it $\pi$-base at $x\in X$} is a collection $\mathcal U$ of non-void
open subsets of $X$ such that every open neighborhood of $x$ in $X$ contains an element of $\mathcal U$.

\par\bigskip\noindent
\begin{thm}\label{thm:shapirovskii}{\rm ({\bf Shapirovskii, \cite[corollary to Theorem 2]{Sha}})}
Any hereditarily normal compact space has a point with countable $\pi$-base.
\end{thm}

\par\bigskip\noindent
\begin{thm}\label{thm:compactainT5groups}
Every compact subset of a hereditarily normal topological group is metrizable.
\end{thm}
\begin{proof}
Let $G$ be a hereditarily normal topological group and let $X$ be its compact subset.
We may assume that $X$ is infinite.
By virtue of Theorem \ref{thm:T5groupwithsequence} it suffices to find a non-trivial convergent sequence in $G$.
Assume that no such sequences exist in $G$.

\par\bigskip\noindent
{\it Claim 1. 
If $Z$ is an infinite compact subspace of $G$ then its derived set $Z'$ has no isolated points.
}
\par\smallskip\noindent
If $Z'$ has an isolated point $p$ then $p$ is the limit of a convergent sequence from
$Z\setminus Z'$. Since no such sequences exist, the claim is proved.

\par\bigskip\noindent
{\it Claim 2.
There exist separable compact subsets $A$ and $B$ of $G$ such that $A\cap B=\{e\}$ and $e$ is a limit point for both $A$ and $B$.
}
\par\smallskip\noindent
Let $Z$ be an infinite separable compact subset of $X$.
By Claim 1, the derived set $Z'$ is an infinite compactum without isolated points.
By Shapirovskii's theorem there exists an element in $Z'$ that has a countable $\pi$-base in $Z'$.
By homogeneity of $G$, we can find such $Z$ with an additional requirement
that the neutral element  $e$ is in  $Z'$ and has a countable $\pi$-base in $Z'$.
Fix a collection $\{P_n:n\in\omega\}$ of subsets of $Z'$ that form a $\pi$-base at $e$ in $Z'$.
Put $F_n = \overline P_n\overline P^{-1}_n$. 
\par\medskip\noindent
\begin{description}
	\item[\it Subclaim] {\it $\{F_n:n\in \omega\}$ is a network at $e$ consisting of compact sets, and $e$ is a limit point for every $F_n$.}
\par\smallskip\noindent	
The set $F_n$ is compact because $\overline P_n$ is a closed subset of $Z'$, which is compact.
Thus, $F_n$ is the image of a compactum under the continuous map $\star$. To show that $F_n$'s form a network,
fix any neighborhood $U$ at $e$. One can find an open neighborhood $V$ of $e$ such that
$VV^{-1}\subset U$.  Then there exists an element $P_n$ of our $\pi$-base such that $\overline P_n\subset V$. We then have 
$e\in F_n=\overline P_n\overline P_n^{-1}\subset V_nV_n^{-1}\subset U$.
\par\smallskip\noindent
Finally, to show that $e$ is a limit point for $F_n$, recall that $P_n$ is a non-void open set of $Z'$ and $Z'$ has no isolated points. Pick any $p\in P_n$. Then $P_np^{-1}$ contains $e$ and is a subset of $F_n$. 
Clearly $e$ is a limit point for $P_np^{-1}$. The proof of the subclaim is complete.
\end{description}
\par\medskip\noindent
We are now ready to  construct the desired sets $A$ and $B$. For this we consider two cases.
\begin{description}
	\item[\it Case 1] {\it The assumption of this case is that $F_0\cap ... \cap F_n\not = \{e\}$ for all $n\in \omega$.}
	Then for each $n\in \omega$, we can pick $x_n\in F_0\cap ... \cap F_n$ which is distinct from $e$. Clearly, $x_n\to e$, which contradicts our  assumption that $G$ has no non-trivial convergent sequences.
	
	\item[\it Case 2] {\it The assumption of this case is a failure of  Case 1}. Then there exists the smallest
	$n$ such that $e$ is not a limit point for $F_0\cap ...\cap F_{n+1}$. Then there exists a neighborhood  $U$ of $e$ such that
	$\overline U\cap F_0\cap ... \cap F_{n+1}=\{e\}$. Put $C = \overline U \cap F_0\cap ...\cap F_n$
	and $D=\overline U\cap F_{n+1}$. We have $C\cap D=\{e\}$. By the property of $n$ and Subclaim,
	$e$ is a limit point for both $C$ and $D$. To finish the proof of existence of sets with desired properties 
	it suffices to place $C$ and $D$  in separable compact subsets of $G$ whose only common element is $e$. For this
	recall that
	$F_0,...,F_{n+1}\subset ZZ^{-1}$ and $Z$ is a separable compactum. Further, $(ZZ^{-1})\setminus\{e\}$
	is normal. Therefore, there exist open neighborhoods $V$ and $W$ of $C$ and $D$, respectively, in $ZZ^{-1}$ whose
	closures in $ZZ^{-1}$ have only one point in common, namely, $e$.
	Put $A=\overline V$ and $B=\overline W$. Since $Z$ is separable, $A$ and $B$ are separable. The rest of the desired
	properties are obvious. This completes Case 2 and proves the claim.
\end{description}
\par\bigskip
Let $A$ and $B$ be as in Claim 2. Recall that our assumption is that $G$ does not have any non-trivial convergent sequence.
Therefore, the fact that $e$ is a limit point of compact sets $A$ and $B$ imply that $e$ has uncountable character
both in $A$ and in $B$.

We will finish our argument by  consecutive applications of 
Lemma \ref{lem:katetov} as follows: Put $S= A$, $T=A$, and $s=t=e$. The sets $S$ and  $T$ together with points $s$ and $t$ satisfy the hypothesis of Lemma \ref{lem:katetov}. Therefore, there exists a compact set $A_1\subset A$ such that (1) and (2) hold:
\begin{description}
	\item[\rm (1)] $t=e$ has uncountable character in $A_1$.
	\item[\rm (2)] $sA_1\subset St$.
\end{description}
Note that (2) implies the following:
\begin{description}
	\item[\rm (3)] $A_1\subset A$.
\end{description}
Next, for our second and last application of Lemma \ref{lem:katetov} we put $S=B$, $T=A_1$, and $s=t=e$.
By (1) and the properties of $B$ listed in Claim 2, these sets and points satisfy the hypothesis of Lemma \ref{lem:katetov}.
Therefore, there exists a compact set $A_2$ such that (4),(5), and (6) hold:
\begin{description}
	\item[\rm (4)] $A_2\subset A_1$.
	\item[\rm (5)] $t=e$ has uncountable character in $A_2$.
	\item[\rm (6)] $sA_2\subset St$
\end{description}
Note that (6) implies the following:
\begin{description}
	\item[\rm (7)] $A_2\subset B$.
\end{description}
By (5), there exists $a\in A_2$ which is distinct from $e$. By (7), $a\in B$. By (3) and (4), $a\in A$.
Therefore,  $a\in A\cap B$ and is distinct from $e$, which contradicts the fact that $A\cap B=\{e\}$. 
This contradiction completes our proof.
\end{proof}

\par\bigskip
Observe that the first paragraph of the proof of Theorem \ref{thm:compactainT5groups} suggests to re-phrase
Theorem \ref{thm:compactainT5groups} in a way that is more descriptive of the internal
structure (with regard to convergence) of hereditarily normal topological groups.
\par\bigskip\noindent
\begin{thm}\label{thm:mainrephrased}
Let $G$ be a hereditarily normal topological group. Then either $G$ has a non-trivial convergent sequence and a $G_\delta$-diagonal,
or $G$ has no non-trivial convergent sequences and every compact subset of $G$ is finite. In either case, every compact subset of $G$ is metrizable.
\end{thm}
\par\bigskip\noindent
\par\bigskip\noindent
{\bf Acknowledgment.} 
The author would like to thank the referee for many helpful remarks, corrections, and suggestions.

\end{document}